\documentclass[12pt]{amsart}

\usepackage{latexsym,amsmath,amssymb,amsthm,amsfonts}
\usepackage{hyperref}

\usepackage{color}

\usepackage[lite,alphabetic]{amsrefs}

\newcommand{\R}{{\mathbb{R}}}
\newcommand{\wt}{\widetilde}

\newcommand{\E}{{\mathcal{E}}}

\renewcommand{\P}{{\mathcal{P}}}

\newcommand{\M}{{\mathfrak{M}}}

\newcommand{\zero}{{\boldsymbol{0}}}

\newcommand{\I}{{\mathcal{I}}}
\newcommand{\J}{{\mathcal{J}}}

\newcommand{\rank}{{\rm rank}}
\newcommand{\lin}{{\rm span}}

\def\br#1{\overline{#1}}
\def\ubr#1{\underline{#1}}

\newcommand{\Kfive}{{K_5-e}}
\newcommand{\wtG}{{\widetilde G}}

\renewcommand{\phi}{\varphi}

\def\be{\begin{equation}}
\def\ee{\end{equation}}

\newtheorem{theorem}{Theorem}[section]
\newtheorem{lemma}[theorem]{Lemma}
\newtheorem*{lemma-repeat}{Lemma~4.2}
\newtheorem{corollary}[theorem]{Corollary}
\newtheorem{proposition}[theorem]{Proposition}
\theoremstyle{remark}
\newtheorem{remark}[theorem]{Remark}

\numberwithin{equation}{section}

\title{Non-existence of $(76,30,8,14)$ strongly regular graph}

\author{A. V. Bondarenko}
\address{Department of Mathematical Sciences, Norwegian University of Science and Technology, NO-
7491 Trondheim, Norway}
\email{andriybond@gmail.com}

\author{A. Prymak}
\address{Department of Mathematics, University of Manitoba, Winnipeg, MB, R3T2N2, Canada}
\email{prymak@gmail.com}

\author{D. Radchenko}
\address{The Abdus Salam International Centre for Theoretical Physics, Str. Costiera 11, 34151 Trieste, Italy} \email{danradchenko@gmail.com}

\thanks{The first author was supported in part by Grant 227768 of the Research Council of Norway. The third author was supported in part by NSERC of Canada Discovery Grant RGPIN 04863-15. All authors were supported in part by NSERC of Canada Discovery Grant RGPIN 372001-10, including the visits of A.~V.~Bondarenko and D.~Radchenko to the University of Manitoba in April 2013.}



\begin{document}

\begin{abstract}
We prove the non-existence of strongly regular graph with parameters $(76,30,8,14)$. We use Euclidean representation of a strongly regular graph together with a new lower bound on the number of 4-cliques to derive strong structural properties of the graph, and then use these properties to show that the graph cannot exist.
\end{abstract}

\maketitle

\section{Introduction}




Let $G=(V,E)$, where $V$ is the set of vertices and $E$ is the set of edges, be a finite, undirected, simple graph. The graph $G$ is \emph{strongly regular} with parameters $(v,k,\lambda,\mu)$ if $G$ is $k$-regular on $v$ vertices, any two adjacent vertices have $\lambda$ common neighbors, and any two non-adjacent vertices have $\mu$ common neighbors. It is not known in general for which parameters $(v,k,\lambda,\mu)$ strongly regular graphs exist. One can easily deduce certain necessary conditions on the parameters (see Section~\ref{prel}), but the general pattern is still far from being understood, see~\cite{Br-www} for a list of results for $v\le 1300$. Our main result is the following theorem.
\begin{theorem}\label{main}
There is no strongly regular graph with parameters $(76,30,8,14)$.
\end{theorem}
Some numerical evidence for non-existence of this graph was given in~\cite{Deg}*{Section~6.1.6, p.~204}, which involved a significant but not exhaustive computer search.

Let us outline the structure of the proof. Assuming the existence of such a graph $G$, we first show that it must contain a $4$-clique (complete graph on $4$ vertices) as a subgraph. This is a crucial first step, which then allows to show that $G$ contains a much larger ``nice'' induced subgraph: either a $(40,12,2,4)$ strongly regular graph, or a $16$-coclique (empty graph on $16$ vertices), or a complete bipartite graph $K_{6,10}$ (two parts of $6$ and $10$ vertices, with an edge between vertices if and only if the vertices are from different parts). In what follows, by a subgraph we always mean the induced subgraph. Each of these three cases is treated differently but ultimately leads to a contradiction. The last two cases were completed using machine-assisted searches with total running time of under two hours on a personal computer. We would like to emphasize that our methods establish strong structural properties of the graph, and use of computer is minor.

To establish such strong structural properties of $G$, we heavily use the Euclidean representation of a strongly regular graph as a system of points on the unit sphere in a Euclidean space (see Section~\ref{prel} for the definitions). While our arguments may be applied for any strongly regular graph, we observed non-trivial corollaries mostly for graphs which have $2$ as an eigenvalue. 

While Euclidean representation does provide a system of points in a finite-dimensional space, that dimension does not need to be small. Understanding the structure of such point sets can be a challenging task. An important part of our approach is the use of orthogonal projection of the points from Euclidean representation onto a subspace of small dimension, such as $\R^2$ or $\R^3$.

Another result of possibly independent interest is a lower bound on the number of $4$-cliques in a strongly regular graph, see Theorem~\ref{K4thm}. 
The proof is based on the fact that reproducing kernels are positive definite and has the same spirit as the Krein's bound and the absolute bound on parameters of a SRG.


The paper is organized as follows. We describe some preliminaries and notations in Section~\ref{prel}. Then we establish our lower bound on the number of $4$-cliques (Theorem~\ref{K4thm}) in Section~\ref{K4bound}. 
In Section~\ref{sect-reduction}, we reduce Theorem~\ref{main} to one of the three main cases, which are treated in Sections~\ref{big-srg}, \ref{empty16}, and~\ref{K6,10}. For reader's convenience, we provide in~\cite{BPR-web} the scripts for computer searches required for the proofs in Sections~\ref{empty16} and~\ref{K6,10}, and several functions implemented in SageMath (\cite{SAGE}) computer algebra system which can help verify some technical computations. However, all the proofs in this paper are self-contained and do not depend on~\cite{BPR-web}.

%


\section{Preliminaries}\label{prel}
Throughout this section 
let $G=(V,E)$ be a strongly regular graph (SRG) with parameters $(v,k,\lambda,\mu)$ (we sometimes say that $G$ is a $(v,k,\lambda,\mu)$ SRG). By $N(i):=\{j:(i,j)\in E\}$ we will denote the set of all neighbors of a vertex $i\in V$.

\subsection{Spectral properties}
The incidence matrix $A$ of $G$ has the following properties:
\be\label{matrix-equation}
AJ = kJ,\quad\text{and}\quad
A^2 + (\mu - \lambda)A + (\mu - k)I = \mu J,
\ee
where $I$ is the identity matrix and $J$ is the matrix with all
entries equal to~$1$. These conditions imply that
\begin{equation}\label{par_restriction}
(v - k - 1)\mu = k(k - \lambda - 1).
\end{equation}
Moreover, the matrix $A$ has only three eigenvalues: $k$ of multiplicity
$1$, a positive eigenvalue
\begin{equation}\label{def_r}
r=\frac 12\left(\lambda-\mu+\sqrt{(\lambda-\mu)^2+4(k-\mu)}\right)
\end{equation}
of multiplicity
\begin{equation}\label{def_f}
f=\frac 12
\left(v-1-\frac{2k+(v-1)(\lambda-\mu)}{\sqrt{(\lambda-\mu)^2+4(k-\mu)}}\right),
\end{equation}
and a negative eigenvalue
\begin{equation}\label{def_s}
s=\frac 12\left(\lambda-\mu-\sqrt{(\lambda-\mu)^2+4(k-\mu)}\right)
\end{equation}
of multiplicity
\begin{equation}\label{def_g}
g=\frac 12
\left(v-1+\frac{2k+(v-1)(\lambda-\mu)}{\sqrt{(\lambda-\mu)^2+4(k-\mu)}}\right).
\end{equation}
Clearly, $f$ and $g$ should be integers. This together
with~\eqref{par_restriction} gives a set of strong conditions on the parameters
$(v,k,\lambda,\mu)$ for strongly regular graphs. The reader can refer to~\cite[Section~9.1.5]{BrHa} for the proofs of the above relations.

For $(v,k,\lambda,\mu)=(76,30,8,14)$, we have $r^f=2^{57}$ and $s^g=(-8)^{18}$.

\subsection{Euclidean representation}\label{eq_rep}
Now we will construct a Euclidean representation of $G$ in $\R^g$. Take the columns
$\{y_i: i\in V\}$ of the matrix $A-rI$ and let $x_i:=z_i/\|z_i\|$,
where
\[
z_i=y_i-\frac 1{|V|}\sum_{j\in V}y_j, \quad i\in V,
\]
and $\|z_i\|:=(z_i \cdot z_i)^{1/2}$. Here and below $x \cdot y$ will denote the dot product of $x$ and $y$ in the corresponding Euclidean space, and $|V|$ denotes the number of elements in a set $V$. It is straightforward to verify that $\rank(\lin\{x_i:i\in V\})=g$ (so $x_i$ can be assumed to be elements of $\R^g$) and the following two properties are satisfied. First, there are only two possible non-trivial values of the dot product depending on adjacency:
\begin{equation}\label{edges_pq}
x_i\cdot x_j=
\begin{cases}1, & \text{if }i=j, \\
p, & \text{if }i\text{ and }j\text{ are adjacent},\\
q, & \text{otherwise},
\end{cases}
\end{equation}
where $p$ and $q$ are real numbers from the interval $(-1,1)$, namely
\begin{equation}\label{pq_def}
p=s/k, \quad\text{and}\quad
q=-(s+1)/(v-k-1).
\end{equation}
The second property is that the set $\{x_i:i\in V\}$ forms a spherical
$2$-design, i.e.,
\begin{equation}\label{design}
\sum_{i\in V}x_i={0},\quad\text{and}\quad
\sum_{i\in V}(x_i\cdot y)^2 =\frac{|V|}{g} \quad\text{for any }y,\ \|y\|=1.
\end{equation}
For more information on the relations between the
Euclidean representation of strongly regular graphs and spherical
designs see, e.g., \cite{Cam}.

One of the key facts that we will use  is the following evident proposition.
\begin{proposition}\label{prop-rank} Let $G=(V,E)$ be a SRG with parameters $(v,k,\lambda,\mu)$, and $x_i$, $i\in V$, be its Euclidean representation in $\R^g$. Then for any subset $U\subset V$, the Gram matrix
$\left(x_i\cdot x_j\right)_{i,j\in U}$ is non-negative definite and its rank equals to  $\rank(\lin\{x_i: i\in
U\})$, which is at most $g$. If $A$ is the adjacency matrix of the subgraph induced by $U$, then $\left(x_i\cdot x_j\right)_{i,j\in U}=pA+I+q(J-I-A)$.
\end{proposition}

Another observation that we will use is that
\be\label{vertex-through-neighbors}
x_i=\frac{1}{kp}\sum_{j\in N(i)}x_j \quad\text{for each }i\in V.
\ee
Indeed, for arbitrary $l\in G$, it straightforward to check that $(kpx_i-\sum_{j\in N(i)}x_j)\cdot x_l=0$ (using~\eqref{par_restriction}, \eqref{def_s}, \eqref{edges_pq} and \eqref{pq_def}).

\begin{remark}\label{complement}
One can construct a dual Euclidean representation of $G$ in $\mathbb{R}^f$ which will possess similar properties. This can be done by considering the complement of $G$, which is a strongly regular graph with parameters $(v,v-1-k,v-2k+\mu-2,v-2k+\lambda)$; then $f$ and $g$ interchange.
\end{remark}

For $(v,k,\lambda,\mu)=(76,30,8,14)$, the Euclidean representation in $\R^{18}$ has dot products $(p,q)=(-\frac4{15},\frac7{45})$, and the Euclidean representation in $\R^{57}$ (obtained through the complement) has dot products $(p,q)=(\frac1{15},-\frac1{15})$, see~\eqref{edges_pq} and~\eqref{pq_def}.

For a subset $A$ of vertices of $G$, it will be convenient to denote $\br{A}:=\sum_{i\in A}x_i$ where $x_i\in \R^g$ is the Euclidean representation of $i$. In the same manner, we denote $\ubr{A}:=\sum_{i\in A}z_i$ where $z_i\in \R^f$ is the dual Euclidean representation of $i$.

\subsection{Rank of Gram matrix for certain subgraphs in a $(76,30,8,14)$ SRG}
\begin{lemma}\label{new-rank-lemma}
Let $\wt G$ be an induced subgraph of a $(76,30,8,14)$ strongly regular graph $G$, and $B(\wt G)$ be the Gram matrix of vectors $x_i$, $i\in\wt G$, where $x_i\in\R^{18}$ denotes Euclidean representation of vertex $i\in G$.
\begin{itemize}
\item[(i)] If $\wt G$ is a $(40,12,2,4)$ strongly regular graph, then $\rank(B(\wt G))=16$.
\item[(ii)] If $\wt G$ is a $16$-coclique, then $\rank(B(\wt G))=16$.
\item[(iii)] If $\wt G$ is a $K_{6,10}$, then $\rank(B(\wt G))=15$.
\item[(iv)] If $\wt G$ is a disjoint union of $n$ cycles on $20$ vertices, then $\rank(B(\wt G))=21-n$.
\end{itemize}
\end{lemma}
\begin{proof}
If $A$ is the adjacency matrix of $\wt G$, then $B(\wt G)=pA+I+q(J-I-A)$ by~\eqref{edges_pq}, where $p=-\frac{4}{15}$ and $q=\frac{7}{45}$. This allows to compute the spectrum of $B(\wt G)$ from the spectrum of $A$ for~(i), (ii), and (iv), while the spectrum of $B(\wt G)$ can be computed directly for~(iii).
%
%
%
\end{proof}

\section{Lower bound on the number of $4$-cliques}\label{K4bound}

We begin with some preliminaries from harmonic analysis.

\subsection{Spherical harmonic polynomials}\label{harmonic}
A homogeneous real  polynomial of degree $t$ on $\R^n$ is a real linear combination of monomials $x_1^{t_1}\dots x_n^{t_n}$, where $t_1, \dots, t_n$ are non-negative integers with sum $t$.
Let $\Delta$ be the Laplace operator in $\mathbb{R}^{n}$
\[
\Delta =\sum_{j=1}^{n}\frac{\partial^2}{\partial x_j^2}.
\]
An  polynomial $P$ on $\mathbb{R}^{n}$ is said to be harmonic if
$\Delta P=0$. For integer $t\ge 1$, the restriction to the unit sphere $S^{n-1}$ in $\R^n$ of a
homogeneous harmonic polynomial of degree $t$ is called a
spherical harmonic of degree $t$. The vector space of all
spherical harmonics of degree $t$ will be denoted by
$\P_{n,t}$. Various properties of spherical harmonics can be found, for example, in~\cite[Chapter~1]{DaXu}.

We equip $\P_{n,t}$ with the inner product
\[
\langle P,Q \rangle = \int_{S^{n-1}} P(x)Q(x)\,d\mu_n(x),
\]
where $\mu_n$ is the Lebesgue measure on $S^{n-1}$ normalized by $\mu_n(S^{n-1})=1$. By the Riesz representation theorem, for each point $x\in S^{n-1}$ there exists a unique polynomial $P_x\in\P_{n,t}$ satisfying
\[
\langle P_x, Q \rangle = Q(x) \quad\text{for all}\quad Q\in\P_{n,t}.
\]
This spherical harmonic $P_x$ can be conveniently expressed using the Gegenbauer polynomials $C_t^{(\alpha)}(\xi)$ with $\alpha=(n-2)/2$.  The polynomials $C_t^{(\alpha)}(\xi)$ are orthogonal on $[-1,1]$ with the weight $(1-\xi^2)^{\alpha-1/2}$, and can be defined by the
generating function
\[
\frac{1-z^2}{(1-2\xi z+z^2)^{\alpha+1}}=\sum_{t=0}^\infty \frac{t+\alpha}{\alpha} C_t^{(\alpha)}(\xi)z^t,
\]
or in many other ways~\cite[Appendix~B.2]{DaXu}. Now, for $x,y\in S^{n-1}$, we have
(see, e.g., \cite[Lemma~1.2.5, Theorem~1.2.6]{DaXu}):
\[
\langle P_x, P_y \rangle=Z_{n,t}(x \cdot y),
\quad\text{where}\quad Z_{n,t}(\xi)=\frac{2t+n-2}{n-2}C_t^{((n-2)/2)}(\xi).
\]
Note that $\langle P_x, P_y \rangle$ depends only on $x\cdot y$, which also easily follows from the fact that the space $\P_{n,t}$ is rotation invariant. The spherical harmonic $Z_{n,t}(x\cdot y)$ (with fixed $x\in S^{n-1}$ as a function of $y\in S^{n-1}$) is referred to as a zonal harmonic.

Using the Cauchy-Schwarz inequality in $\P_{n,t}$, for any finite sets of points $\{x_i\}_{i\in \I}$ and $\{y_j\}_{j\in\J}$ from $S^{n-1}$, we obtain
\begin{align*}
\left( \sum_{i\in\I,j\in\J} \langle P_{x_i}, P_{y_j} \rangle \right)^2
&= \left\langle \sum_{i\in\I} P_{x_i}, \sum_{j\in\J} P_{y_j} \right\rangle^2  \\
&\le \left\langle \sum_{i\in\I} P_{x_i}, \sum_{i\in\I} P_{x_i} \right\rangle
\left\langle \sum_{j\in\J} P_{y_j}, \sum_{j\in\J} P_{y_j} \right\rangle \\
&=\sum_{i,i'\in\I}\langle P_{x_i}, P_{x_{i'}} \rangle \sum_{j,j'\in\J}\langle P_{y_j}, P_{y_{j'}} \rangle.
\end{align*}
Rewriting this in terms of the polynomials $Z_{n,t}$, we obtain (recall that  $x_i,y_j\in S^{n-1}$)
\be\label{gegen}
\left( \sum_{i\in\I,j\in\J} Z_{n,t}(x_i\cdot  y_j) \right)^2
\le
\left( \sum_{i,i'\in\I} Z_{n,t}(x_i\cdot  x_{i'}) \right)
\left( \sum_{j,j'\in\J} Z_{n,t}(y_j\cdot  y_{j'}) \right).
\ee
This inequality with $t=4$ and proper choice of $x_i$, $y_j$ arising from the Euclidean representation of a strongly regular graph will play a crucial role in the next subsection.



\begin{remark}
The inequality~\eqref{gegen} is valid whenever the function $Z_{n,t}$ is positive definite in $S^{n-1}$ in terminology of~\cite{Sch}. Any finite positive linear combination of Gegenbauer polynomials $C_t^{((n-2)/2)}$ (with fixed $n$ and different $t$) is positive definite in $S^{n-1}$. On the other hand, any positive definite function in $S^{n-1}$ is a series of Gegenbauer polynomials with non-negative coefficients, see~\cite[Theorem~1]{Sch}.
\end{remark}

\subsection{The $K_4$ bound}

Let $G=(V,E)$ be a strongly regular graph with parameters
$(v,k,\lambda,\mu)$. Recall that for any vertex $x\in V$ we define $N(x)$ to be the set
of all neighbors of $x$. Also let $N'(x)$ be the set of non-neighbors
of $x$, i.e. $N'(x)=V\setminus(\{x\}\cup N(x))$. For any
adjacent vertices $x$ and $y$, we consider the following vertex partition of $V\setminus\{x,y\}$
\[
\{G_1,\dots,G_4\}:=\{N(x)\cap N(y),N'(x)\cap N(y),N(x)\cap N'(y),N'(x)\cap N'(y)\}.
\]
Let $\E_\pi=(a_{i,j})_{i,j=1}^4$, where $a_{i,j}$ is the number of edges $(x',y')\in E$ such that $x'\in G_i$ and $y'\in G_j$. The following statement expresses
all entries of $\E_\pi$ in terms of the parameters of a strongly regular graph and the value
of $a_{1,1}$. 
\begin{proposition}\label{edge_count_k4}
With the above notations, let $a:=a_{1,1}$. We have (the values below main diagonal are omitted)
\[
\E_\pi=
\left(
\begin{smallmatrix}
a & \lambda(\lambda-1)-2a & \lambda(\lambda-1)-2a & \lambda(k-2\lambda)+2a \\
 & \frac{\lambda(k-2\lambda)}2+a & (\mu-1)(k-\lambda-1)-\lambda(\lambda-1)+2a & (k-\mu)(k-\lambda-1)-\lambda(k-2\lambda)-2a \\
 & & \frac{\lambda(k-2\lambda)}2+a & (k-\mu)(k-\lambda-1)-\lambda(k-2\lambda)-2a \\
 & & & \frac{k(v-2k+\lambda)}2-(k-\mu)(k-\lambda-1)+\frac{\lambda(k-2\lambda)}2+a
\end{smallmatrix}
\right).
\]
\end{proposition}
We omit the proof which consists of standard combinatorial arguments that use strong regularity of the graph and counting of appropriate paths of length two.

To derive a bound on the number of $4$-cliques we will use~\eqref{gegen}, where we choose $x_i\in\R^g$ to be the Euclidean representation of $i\in V$ (satisfying~\eqref{pq_def}) for all $|V|=v$ vertices of the graph, and $y_j:=\frac{x_{j_1}+x_{j_2}}{\|x_{j_1}+x_{j_2}\|}$ for all $|E|=\frac{vk}2$ edges $j\in E$, here $j$ joins the vertices $j_1,j_2\in V$. Note that $\|x_{j_1}+x_{j_2}\|=\sqrt{2+2p}$. We proceed by computing and introducing notations for certain components of~\eqref{gegen}. Note that the variable $n$ of~\eqref{gegen} is now equal to $g$.

Fixing a vertex $i\in V$, we have three possibilities: $i'=i$, $i'\in N(i)$, or $i'\in N'(i)$. Thus
\be\label{fla-vv}
\sum_{i,i'\in V}Z_{g,t}(x_i\cdot x_{i'})=v(Z_{g,t}(1)+kZ_{g,t}(p)+(v-k-1)Z_{g,t}(q))=:\Psi_A(v,k,\lambda,\mu,t).
\ee
Next, there are $k$ edges which join $i$ and a vertex in $N(i)$. There are $\frac{k\lambda}2$ edges joining some two vertices of $N(i)$. Next, some $(v-k-1)\mu$ edges are between $N(i)$ and $N'(i)$. Finally, we have $\frac{(v-k-1)(k-\mu)}{2}$ edges in $N'(i)$. Thus, we obtain
\begin{gather}
\sum_{i\in V,j\in E}Z_{g,t}(x_i\cdot y_j) =vkZ_{g,t}\left(\frac{1+p}{\sqrt{2+2p}}\right)+\frac{vk\lambda}2Z_{g,t}\left(\frac{2p}{\sqrt{2+2p}}\right)
+v(v-k-1)\mu Z_{g,t}\left(\frac{p+q}{\sqrt{2+2p}}\right)
\nonumber \\
+\frac{v(v-k-1)(k-\mu)}{2}Z_{g,t}\left(\frac{2q}{\sqrt{2+2p}}\right)
=:\Psi_B(v,k,\lambda,\mu,t).
\label{fla-ve}
\end{gather}
If $j\in E$ joins $x,y\in V$, we denote by~$n_j$ the number of edges in $N(x)\cap N(y)$. For a fixed $j\in E$, by considering various cases for $j'\in E$ and using Proposition~\ref{edge_count_k4}, we obtain that the expression $\sum_{j'\in E}Z_{g,t}(y_j\cdot y_{j'})$ is a linear function of~$n_j$, whose coefficients depend only on the graph parameters and on~$t$. Clearly, $\sum_{j\in E}n_j=6N$, where $N$ is the number of $4$-cliques in~$G$. Therefore,
\begin{gather} \label{fla-linterm}
\sum_{j,j'\in E}Z_{g,t}(y_j\cdot y_{j'})=:\Psi_{C_0}(v,k,\lambda,\mu,t)+N\Psi_{C_1}(v,k,\lambda,\mu,t),
\end{gather}
where the leading coefficient is given by
\begin{equation}
\label{fla-ee1}
\Psi_{C_1}(v,k,\lambda,\mu,t) = 6 \sum_{l=0}^4 (-1)^l \binom 4l Z_{g,t}\left(\frac{(4-l)p+lq}{2+2p}\right).
\end{equation}
Now, the inequality~\eqref{gegen} immediately implies the following result.
\begin{theorem}\label{K4thm}
Let $N$ be the number of $4$-cliques in a strongly regular graph with parameters $(v,k,\lambda,\mu)$. Then for any positive integer $t$ one has
\[
(\Psi_B(v,k,\lambda,\mu,t))^2\le \Psi_A(v,k,\lambda,\mu,t)\left(\Psi_{C_0}(v,k,\lambda,\mu,t)+N\Psi_{C_1}(v,k,\lambda,\mu,t)\right),
\]
where $\Psi_A$, $\Psi_B$, $\Psi_{C_0}$ and $\Psi_{C_1}$ are defined by~\eqref{fla-vv}, \eqref{fla-ve}, and~\eqref{fla-linterm}. 
\end{theorem}
For our applications, we choose $t=4$. 
In this case the resulting bound on $N$ can be expressed in terms of a rational function of $k$, $r$, $s$ of degree $\le 10$ in each variable (here $r$ and $s$ are the corresponding eigenvalues, see~\eqref{def_r} and~\eqref{def_s}). The expression for this rational function is quite lengthy and is provided in~\cite{BPR-web}, where one can also find a table of non-trivial bounds on $N$ for all admissible $v\le 1300$.
We also include a part of this table below to illustrate the result for some small parameters of strongly regular graphs.
\begin{center}\begin{tabular}{l|r}
	\hline SRG parameters $(v,k,\lambda, \mu)$ & Lower bound on $\#K_4$\ \\
	\hline
	(75, 32, 10, 16) & 783 \\
	(76, 30, 8, 14) &  39 \\
	(95, 40, 12, 20) &  1827 \\
	(147, 66, 25, 33) &  58833 \\
	(148, 63, 22, 30) &  34850 \\
	(154, 72, 26, 40) &  58458 \\
	(169, 70, 27, 30) &  12744 \\
	(171, 60, 15, 24) &  3645 \\
	(176, 70, 24, 30) &  34168 \\
	\hline
\end{tabular}\end{center}
\smallskip


The following is an immediate corollary of Theorem~\ref{K4thm} that we need for the proof of Theorem~\ref{main}.
\begin{corollary}\label{K4inour}
Any $SRG(76,30,8,14)$ contains a $K_4$.
\end{corollary}
More precisely, the bound from Theorem~\ref{K4thm} provides us with $N\ge\frac{2128}{55}$, so $N\ge 39$.
In this case, in~\eqref{fla-vv}--\eqref{fla-ee1} we have $Z_{g,t}(\xi)=Z_{18,4}(\xi) = 54 - 2160 \xi^2 + 7920 \xi^4$.

\section{Reduction to $SRG(40,12,2,4)$ or $16$-coclique or $K_{6,10}$ as a subgraph}\label{sect-reduction}

Theorem~\ref{main} follows immediately from the next four lemmas. Recall that $N(z)$ and $N'(z)$ are the sets of neighbors and non-neighbors of a vertex $z$, respectively.

\begin{lemma}\label{reduction}
If $G$ is a $SRG(76,30,8,14)$, then there is a subgraph $\wt G$ of $G$ satisfying one of the following statements:\\
(i) $\wt G$ is a $SRG(40,12,2,4)$, and for any $z\in G\setminus\wt G$ both $N(z)\cap \wt G$ and $N'(z)\cap \wt G$ are $4$-regular subgraphs on $20$ vertices, and $|N(z_1)\cap N(z_2)\cap\wt G|=8$ for any adjacent $z_1,z_2\in G\setminus\wt G$; or\\
(ii) $\wt G$ is a $16$-coclique; or\\
(iii) $\wt G$ is a $K_{6,10}$.
\end{lemma}

Recall that $n$-coclique is a graph with $n$ vertices without edges, and $K_{m,n}$ is the complete bipartite graph.

\begin{lemma}\label{srg40}
If $G$ is a $SRG(76,30,8,14)$, there is no induced subgraph $\wt G\subset G$ which is a $SRG(40,12,2,4)$, and such that for any $z\in G\setminus\wt G$ both $N(z)\cap \wt G$ and $N'(z)\cap \wt G$ are $4$-regular subgraphs on $20$ vertices, and $|N(z_1)\cap N(z_2)\cap\wt G|=8$ for any adjacent $z_1,z_2\in G\setminus\wt G$.
\end{lemma}

\begin{lemma}\label{16coclique}
$16$-coclique cannot be an induced subgraph of $SRG(76,30,8,14)$.
\end{lemma}

\begin{lemma}\label{K610}
$K_{6,10}$ cannot be an induced subgraph of $SRG(76,30,8,14)$.
\end{lemma}

In this section we will prove Lemma~\ref{reduction} only.

\begin{proof}[Proof of Lemma~\ref{reduction}.]

Let $G$ be a $(76,30,8,14)$ strongly regular graph. If $H$ is an induced subgraph of $G$, $m=|H|$, we define
\begin{align}\label{dj_def}
d_j&:=|\{x\in H:\text{there are exactly }j\text{ edges from }x\text{ to vertices in }H\}|\\
b_j&:=|\{x\in G\setminus H:\text{there are exactly }j\text{ edges from }x\text{ to vertices in }H\}|.\label{bj_def}
\end{align}
Using strong regularity of $G$, it is straightforward to obtain the equations
\begin{align}
\label{l10}
\sum_{j\ge0} b_j &= 76-m, \\
\label{l11}
\sum_{j\ge0} j b_j &= 30m-\sum_{j\ge0} jd_j, \quad\text{and} \\
\label{l12}
\sum_{j\ge0} \binom j2 b_j &= 14 \binom m2 - \sum_{j\ge0} \binom j2 d_j - 3\sum_{j\ge0}jd_j.
\end{align}

If $H$ is a $K_4$, then $(d_j)_{j\ge0}=(0,0,0,4,0,\dots)$, and the above equations become  $\sum_{j\ge0} b_j = 72$, $\sum_{j\ge0} j b_j = 108$, and $\sum_{j\ge0} {j \choose 2} b_j = 36$, which can be combined to obtain $\sum_{j\ge0} (j-1)(j-2) b_j = 0$, so that $b_j = 0$ unless $j=1,2$, and then $b_1=b_2=36$. As $b_3=b_4=0$, $G$ cannot contain $K_5$ or $\Kfive$ as a subgraph, where $\Kfive$ denotes a $K_5$ with one edge removed. In what follows this fact will be used several times.

By Corollary~\ref{K4inour}, there is a $K_4$ in $G$, which we denote by $G_0$. For $j=1,2$, let $G_j$ be the subgraph of $G\setminus G_0$ such that each vertex of $G_j$ has exactly $j$ neighbors in $G_0$. Note that $G$ is partitioned into $G_0$, $G_1$, and $G_2$, and, as established in the previous paragraph, $|G_1|=|G_2|=36$. Further, strong regularity of $G$ implies the following: $G_1$ and $G_2$ are regular of degree $11$ and $10$, respectively; any vertex of $G_0$ has exactly $9j$ neighbors in $G_j$, $j=1,2$; any vertex of $G_j$ has precisely $18$ neighbors in $G_{3-j}$, $j=1,2$. Now we consider two cases depending on whether $G_2$ contains a triangle.

{\bf Case 1. $G_2$ contains no triangles.} Then we will show that $\wtG := G\setminus G_2$ is $SRG(40,12,2,4)$ and that other part of statement~(i) of Lemma~\ref{reduction} is satisfied. For a vertex $i\in G$, let $x_i\in\R^{18}$ be the image of $x$ under the Euclidean representation of $G$. Recall that in this case the dot product $x_i\cdot x_j$ is either $-\frac{4}{15}$ or $\frac{7}{45}$ when $i$ and $j$ are adjacent or not adjacent, respectively (see~\eqref{edges_pq}), and that for a subset $A$ of vertices of $G$, we set $\br{A}:=\sum_{i\in A}x_i$. For $x\in \wtG$, we have $|N(x)\cap G_2|=18$, and let $w$ be the number of edges in $N(x)\cap G_2$. The Gram matrix of $\br{G_0}$, $\br{N(x)\cap G_2}$, $\br{\{x\}}$ is 
(we omit the values below the main diagonal)
\begin{equation}\label{gg1}
\begin{pmatrix}
\frac45 & -4 & \frac15 \\
 & \frac{-38 w}{45} & \frac{-24}5 \\
 & & 1
\end{pmatrix},
\end{equation}
whose determinant is $\frac{722}{1125}(36-w)\ge0$. But since $G_2$ has no triangles,
the average value of $w$ over all vertices in $\wt G$  is $36$, so we always have $w=36$. Then the matrix~\eqref{gg1} is singular, and we find linear dependence
\[
\br{G_0}+\frac14\br{N(x)\cap G_2}+\br{\{x\}}=0.
\]
Multiplying this equation by $\br{\{z\}}$ for $z\in\wtG$, we find that the number of neighbors of $z$ in $N(x)\cap G_2$ is equal to $6$ or $10$ when $z$ is adjacent or is not adjacent to $x$, respectively. This implies that $\wtG$ is $SRG(40,12,2,4)$.

To complete Case~1, it remains to show that the remaining part of statement~(i) is valid, i.e., that for any $z\in G_2$ both $N(z)\cap \wtG$ and $N'(z)\cap \wtG$ are $4$-regular subgraphs on $20$ vertices, and that $|N(z_1)\cap N(z_2)\cap\wtG|=8$ for any adjacent $z_1,z_2\in G_2$. The latter is immediate since $G_2$ has no triangles, so all $8$ common neighbors of $z_1$ and $z_2$ are in $\wtG$. For the former, simple count shows  $|N(z)\cap \wtG|=|N'(z)\cap \wtG|=20$ and that there are $40$ edges in $N(z)\cap \wtG$. Now we let $z_i\in\R^{15}$ be the image of a vertex $i\in\wt G$ under the dual Euclidean representation of $SRG(40,12,2,4)$, in which case the dot product $z_i\cdot z_j$ is either $\frac{1}{6}$ or $\frac{-1}{9}$ when $i$ and $j$ are adjacent or not adjacent, respectively. With $\ubr{N(z)\cap \wtG}:=\sum_{i\in N(z)\cap \wtG}z_i$, we obtain $\ubr{N(z)\cap \wtG}\cdot \ubr{N(z)\cap \wtG}=0$, hence  $\ubr{N(z)\cap \wtG}\cdot z_i=0$ for any $i\in N(z)\cap \wtG$, implying $4$-regularity of $N(z)\cap \wtG$. For $N'(z)\cap \wtG$, the arguments are similar.

{\bf Case 2. $G_2$ has a triangle $G_3$.} There will be several subcases depending on the edge structure between $G_0$ and $G_3$. For each vertex in $G_0$, consider the number of its neighbors in $G_3$, and record the resulting $4$-tuple in descending order. The sum of all entries of such a $4$-tuple is always $6$, and each entry does not exceed $3$. We consider the cases in the reverse lexicographical order of the corresponding $4$-tuples.

{\bf Subcases $(3,3,0,0)$ and $(3,2,1,0)$} are impossible as $G$ cannot contain a $K_5$ or a $\Kfive$ as a subgraph.

{\bf Subcase $(3,1,1,1)$.} We will find a $K_{6,10}$ as a subgraph of $G$.

Take $H=G_0\cup G_3$, then in terms of~\eqref{dj_def}, we have $(d_j)_{j\ge 0}=(0,0,0,0,6,0,1,\dots)$. In terms of~\eqref{bj_def}, using the structure of $G_1$ and $G_2$ and the fact that $G$ has no $K_5$ or $\Kfive$, we get $b_0=b_5=b_6=b_7=0$. Adding the equations~\eqref{l10}, \eqref{l11} and \eqref{l12} with the coefficients $3$, $-2$, and $1$, respectively, we obtain $b_1+b_4=0$, so $b_1=b_4=0$, and the resulting system has only one solution $b_2=27$, $b_3=42$.

Let $y\in G_0$ be the vertex of degree $6$ in $H$. Set $G_4=N(y)\cap (G\setminus H)$, then $|G_4|=24$, and we can decompose $G_4=G_5\cup G_6$, where $G_{3+j}=\{x\in G_4:|N(x)\cap H|=j\}$ for $j=2,3$.
One can compute that $|G_6|=6$ and that the determinant of the Gram matrix of $\br{G_6}$, $\br{H\setminus\{y\}}$, $\br{\{y\}}$ is
$-\frac{1444}{3375}w\ge0$, where $w$ is the number of edges in $G_6$. Therefore $w=0$, the matrix is singular, and
\begin{equation}\label{sub3111-2}
\br{G_6}+4\br{H\setminus\{y\}}+8\br{\{y\}}=0.
\end{equation}
Let $G_7=\{x\in G\setminus (H\cup G_5): |N(x)\cap H|=2\}$, then $|G_7|=b_2-|G_5|=27-18=9$, and each vertex in $G_7$ is non-adjacent to $y$ and has exactly $2$ neighbors in $H\setminus\{y\}$. For any $z\in G_7$, multiplying~\eqref{sub3111-2} by $\br{z}$, we find that $z$ is adjacent to any vertex of $G_6$. Clearly, $y$ is adjacent to all vertices of $G_6$ and not adjacent to any of the vertices of $G_7$. To conclude that the subgraph $G_6\cup G_7\cup\{y\}$ is $K_{6,10}$, it remains to show that there are no edges in $G_7\cup\{y\}$. This is straightforward by considering the determinant of the Gram matrix of $\br{G_7\cup \{y\}}$, $\br{G_6}$.

{\bf Subcase $(2,2,2,0)$.} We will find a $16$-coclique in $G$. Let $G_8=\{x\in G_0: |N(x)\cap (G_0\cup G_3)|=5\}$, then $|G_8|=3$, and the graph $H=G_8\cup G_3$ is $4$-regular on $6$ vertices. As before, we use notations~\eqref{dj_def} and~\eqref{bj_def}, so $(d_j)_{j\ge 0}=(0,0,0,0,6,0,\dots)$. For any $x\in G\setminus H$, let $w=|N(x)\cap H|$, then the determinant of the Gram matrix of $\br{H}$, $\br{\{x\}}$ is $-\frac{1}{2025}(19w-42)^2+\frac{8}{15}\ge0$, providing $1\le w\le3$, i.e., $b_0=b_4=b_5=\dots=0$. The solution of~\eqref{l10}-\eqref{l12} is $(b_j)_{j\ge0}=(0,0,54,16,\dots)$. Now let $G_9=\{x\in G\setminus H:|N(x)\cap H|=3\}$, as we have just found $|G_9|=16$. If $w$ is the number of edges in $G_9$, then the determinant of the Gram matrix of $\br{G_9}$, $\br{H}$ is $-\frac{304}{675}w\ge0$, so $w=0$, and $G_9$ is the required subgraph.

{\bf Subcase $(2,2,1,1)$.} We will find either a $16$-coclique or a $K_{6,10}$ in G. Take $H=G_0\cup G_3$. In notations~\eqref{dj_def} and~\eqref{bj_def}, we have $(d_j)_{j\ge 0}=(0,0,0,0,5,2,\dots)$. For any $x\in G\setminus H$, let $w=|N(x)\cap H|$, then the determinant of the Gram matrix of $\br{H}$, $\br{\{x\}}$ is $-\frac{1}{2025}(19w-42)^2+\frac{13}{15}\ge0$, so $1\le w\le4$, i.e., $b_0=b_5=b_6=\dots=0$. Adding the equations~\eqref{l10}-\eqref{l12} with the coefficients $3$, $-2$, and $1$, respectively, we obtain $b_1+b_4=1$, so $(b_1,b_4)$ is either $(0,1)$ or $(1,0)$. Solving the resulting systems of linear equations, we see that either $(b_j)_{j\ge0}=(0,0,28,40,1,\dots)$ or $(b_j)_{j\ge0}=(0,1,25,43,0,\dots)$.

\emph{First consider the case $(b_j)_{j\ge0}=(0,0,28,40,1,\dots)$}. Let $y\in G\setminus H$ be the vertex with exactly four neighbors in $H$. When $y\in G_2$, each of the two vertices of $G_{10}=N(y)\cap G_0$ has either $1$ or $2$ neighbors in $G_3$, so, using an analogous notation as for subcases, we need to consider one of the three situations $(2,2)$, $(2,1)$, and $(1,1)$.

\emph{For $y\in G_2$ and situation $(2,2)$,} observe that $G_{10}\cup\{y\}\cup G_0$ is $4$-regular on $6$ vertices, so repeating the arguments of the subcase $(2,2,2,0)$, we arrive at existence of a $16$-coclique in $G$.

\emph{For $y\in G_2$ and situation $(2,1)$,} we consider $H'=G_0\cup G_3\cup \{y\}$ and add primes in~\eqref{dj_def}-\eqref{l12} to denote the corresponding quantities and avoid confusion with $H$ fixed at the beginning of the current subcase. Then $(d_j')_{j\ge 0}=(0,0,0,0,3,4,1,\dots)$. For any $x\in G\setminus H'$, let $w=|N(x)\cap H'|$, then the determinant of the Gram matrix of $\br{H'}$, $\br{\{x\}}$ is $-\frac{1}{2025}(19w-56)^2+\frac{2}{3}\ge0$, so $2\le w\le4$, i.e., $b_0'=b_1'=b_5'=b_6'=\dots=0$. The system~\eqref{l10}-\eqref{l12} has unique solution $(b_j')_{j\ge0}=(0,0,7,56,5,\dots)$. For $G_{11}=\{x\in G\setminus H': |N(x)\cap H'|=4\}$, $|G_{11}|=5$, and considering the determinant of the Gram matrix of $\br{G_{11}}$, $\br{H'}$, we see that $G_{11}$ has no edges. Let $y_1$ be the vertex of $G_{10}$ with exactly two neighbors in $G_3$, and $w$ be the number of edges between $y_1$ and $G_{14}$, then the determinant of the Gram matrix of $\br{\{y_1\}}$, $\br{G_{11}}$, $\br{H'\setminus\{y_1\}}$ is
$-\frac{722}{6075}w^2 - \frac{1444}{3645}w\ge0$, so $w=0$. We obtained that $\{y_1\}\cup G_{11}$ is a $6$-coclique, next we wish to find $10$ vertices each connected to all vertices of $\{y_1\}\cup G_{11}$. Let $G_{12}=\{x\in G\setminus H': |N(x)\cap H'|=3\}$, for which $|G_{12}|=b_3'=7$. With $w=0$ in the corresponding Gram matrix, we have linear dependence
\begin{equation}\label{sub2211-2}
5\br{\{y_1\}}+\br{G_{11}}+4\br{H'\setminus\{y_1\}}=0.
\end{equation}
Multiplying this equation by $\br{z}$ for any $z\in G_{12}$, we find that $z$ is adjacent to $y_1$ and to all five vertices of $G_{11}$. It remains to find $3$ more vertices to form the desired $10$. The graph $H'$ has $3$ vertices of degree four, denote them by $G_{13}$. Using $\br{H'\setminus\{y_1\}}=\br{G_{13}}+\br{H'\setminus(\{y_1\}\cup G_{13})}$ in~\eqref{sub2211-2} and multiplying the result by $\br{z}$ for any $z\in G_{13}$, one can see that $z$ is adjacent to all vertices of $\{y_1\}\cup G_{11}$. Using the same argument as in the end of the subcase $(3,1,1,1)$, we obtain that there are no edges in the subgraph $G_{12}\cup G_{13}$, so $\{y_1\}\cup G_{11}\cup G_{12}\cup G_{13}$ is the required $K_{6,10}$ subgraph.

\emph{For $y\in G_2$ and situation $(1,1)$, or for $y\in G_1$,} we also consider $H'=G_0\cup G_3\cup \{y\}$. If $y\in G_1$, then $y$ is adjacent to all vertices of $G_3$ and to one vertex of $G_0$ which has one neighbor in $G_3$ (otherwise we can find a $\Kfive$). Therefore, regardless of whether we have $y\in G_2$ with situation $(1,1)$, or $y\in G_1$, we see that $(d_j')_{j\ge 0}=(0,0,0,0,2,6,\dots)$ and the two vertices of degree $4$ in $H'$ are not adjacent. As the total number of edges in $H'$ is the same as in the situation $(2,1)$ (namely, $19$), we argue similarly to obtain that $b_0=b_1=b_5=b_6=\dots=0$. The system~\eqref{l10}-\eqref{l12} yields the unique solution $(b_j')_{j\ge0}=(0,0,8,54,6,\dots)$. Let $G_{14}=\{x\in G\setminus H':|N(x)\cap H'|=4\}$, then $|G_{14}|=6$ and there are no edges in $G_{14}$ by considering the determinant of the Gram matrix of $\br{G_{14}}$, $\br{H'}$. Moreover, this matrix is singular and one has $\br{G_{14}}+4\br{H'}=0$. Multiplying this equation by $\br{z}$ for every $z\in G_{15}=\{x\in G\setminus H':|N(x)\cap H'|=2\}$, we verify that $z$ is adjacent to all vertices of $G_{14}$. We let the remaining required two vertices $G_{16}$ be the two vertices of $H'$ having degree $4$ in $H'$. Using that they are not adjacent and multiplying $\br{G_{14}}+4\br{H'\setminus G_{16}}+4\br{G_{16}}=0$ by $\br{z}$ for any $z\in G_{16}$, we obtain that $z$ is adjacent to all vertices of $G_{16}$. Arguing as before, $G_{14}\cup G_{15}\cup G_{16}$ is a $K_{6,10}$.

\emph{Now consider the case $(b_j)_{j\ge0}=(0,1,25,43,0,\dots)$.} Define $y\in G\setminus H$ as the vertex with exactly one neighbor in $H$. Let $G_{17}=\{x\in G_2\setminus G_3:|N(x)\cap G_3|=1\}$, then $|G_{17}|=24$ (each vertex of $G_3$ has $8$ neighbors in $G_2\setminus G_3$, with no common neighbors due to $b_4=b_5=0$). Clearly $y\in G_1$, and there are $18$ edges from $y$ to $G_2$, and in particular, at least $24-18=6$ vertices of $G_{17}$ are not adjacent to $y$. Let $G_{18}$ be any such $6$ vertices. The determinant of the Gram matrix of $\br{G_{17}}$, $\br{G_3}$, $\br{G_0}$, $\br{\{y\}}$ is $-\frac{13718}{50625}w\ge0$, so $w=0$, and we obtain linear dependence
\begin{equation}\label{sub2211-3}
\br{G_{17}}+4 \br{G_3}+6\br{G_0}-4\br{\{y\}}=0.
\end{equation}
Note that among $b_3=43$ vertices with exactly $3$ neighbors in $H$, there are $24$ ($G_{17}$) from $G_2$, and hence $19$ from $G_1$. But as $G_1$ is $11$-regular, there are at least $19-11=8$ vertices from these $19$ not adjacent to $y$. Denote any set of such $8$ vertices as $G_{19}$. For any $z\in G_{19}$, multiplying~\eqref{sub2211-3} by $\br{z}$, we get that $z$ has no neighbors in $G_{18}$. Let $G_{20}$ be the subgraph consisting of one vertex that has degree $5$ in $H=G_0\cup G_3$ and is not connected to $y$ (recall that $d_5=2$). Splitting $\br{G_0}=\br{G_0\setminus G_{20}}+\br{G_{20}}$ in~\eqref{sub2211-3} and multiplying the result by $\br{z}$ for $z\in G_{20}$, we obtain that $z$ is not adjacent to any vertex of $G_{18}$.
Considering the determinant of the Gram matrix of $\br{G_{19}}$, $\br{G_3}$, $\br{G_0}$, $\br{\{y\}}$, we obtain that there are no edges in $G_{19}$. To show that $\{y\}\cup G_{18}\cup G_{19}\cup G_{20}$ is a $16$-coclique, it only remains to verify that there are no edges between $G_{19}$ and $G_{20}$. This is straightforward considering the determinant of the Gram matrix of $\br{G_{20}}$, $\br{G_{19}}$, $\br{G_3}$, $\br{G_0\setminus G_{20}}$, $\br{\{y\}}$.
\end{proof}

\section{The case of $SRG(40,12,2,4)$}\label{big-srg}


In this section we prove Lemma~\ref{srg40}. We repeat its statement here for reader's convenience and to remind the additional structure which was obtained in the proof of Lemma~\ref{reduction} and is required here.
\begin{lemma-repeat}
If $G$ is a $SRG(76,30,8,14)$, there is no induced subgraph $\wt G\subset G$ which is a $SRG(40,12,2,4)$, and such that for any $z\in G\setminus\wt G$ both $N(z)\cap \wt G$ and $N'(z)\cap \wt G$ are $4$-regular subgraphs on $20$ vertices, and $|N(z_1)\cap N(z_2)\cap\wt G|=8$ for any adjacent $z_1,z_2\in G\setminus\wt G$.
\end{lemma-repeat}

\begin{proof}[Proof of Lemma~\ref{srg40}.]
Suppose $G$ and $\wt G$ satisfy the conditions of Lemma~\ref{srg40}. By Lemma~\ref{new-rank-lemma}~(i), $\rank B(\wt G)=\rank(\lin(\{x_i,\,i\in\wt G\}))=16$, where $x_i\in\R^{18}$ is the Euclidean representation of $i\in G$. The orthogonal complement of $\lin\{x_i,i\in\wt G\}$ in $\R^{18}$ is a two-dimensional space on which $x_j$, $j\in G\setminus \wt G$, will be orthogonally projected. But first, for $j\in G\setminus \wt G$, denote by $x_j'$ the orthogonal projection of $x_j$ onto $\lin\{x_i,i\in\wt G\}$. We have
\be\label{xj-prime}
x_j'=-\frac19\sum_{i\in N(j)\cap\wt G}x_i+\frac1{18}\sum_{i\in N'(j)\cap\wt G}x_i,
\ee
since $(x_j-x_j')\cdot x_t=0$ for any $t\in\wt G$.

Now fix $j_1,j_2\in G\setminus \wt G$ and partition $\wt G$ into four subgraphs $N(j_1)\cap N(j_2)\cap\wt G$, $N'(j_1)\cap N(j_2)\cap\wt G$, $N(j_1)\cap N'(j_2)\cap\wt G$ and $N'(j_1)\cap N'(j_2)\cap\wt G$. The number of vertices and the number of edges in each of the four subgraphs as well as the number of edges between any two of the four subgraphs can be computed using strong regularity of $\wt G$ and the assumptions of the lemma in terms of only two parameters: $n_{j_1,j_2}:=|N(j_1)\cap N(j_2)\cap\wt G|$ and the number $e_{j_1,j_2}$ of edges in $N(j_1)\cap N(j_2)\cap\wt G$. Using~\eqref{xj-prime}, one can check that
\be\label{xj12-prime}
x_{j_1}'\cdot x_{j_2}'=\frac{19}{270}n_{j_1,j_2}-\frac{52}{81}.
\ee
(The other variable $e_{j_1,j_2}$ cancels out.)

For $j\in G\setminus \wt G$, we denote by $x_j''=x_j-x_j'$ the projection of $x_j$ onto the orthogonal complement of $\lin\{x_i,i\in\wt G\}$. If $j_1=j_2\in G\setminus \wt G$, then $n_{j_1,j_2}=20$, so by~\eqref{xj12-prime} all projections $x_j''$, $j\in G\setminus\wt G$, have the same Euclidean norm, which means they belong to a ($2$-dimensional, planar) circle. For convenience, we use the normalized projections $x_j''':=\frac{x_j''}{\|x_j''\|}$. If $j_1$ and $j_2$ from $G\setminus \wt G$ are adjacent, then $n_{j_1,j_2}=8$, so from~\eqref{xj12-prime} we get $x_{j_1}'''\cdot x_{j_2}'''=-\frac45$. If $j_1$ and $j_2$ from $G\setminus \wt G$ are not adjacent, then from~\eqref{xj12-prime} we find $x_{j_1}'''\cdot x_{j_2}'''=-\frac3{10}n_{j_1,j_2}+\frac{17}5$. Moreover, we claim that if $j_1$ and $j_2$ from $G\setminus \wt G$ are not adjacent, then $x_{j_1}'''\cdot x_{j_2}'''$ is either $1$ or $-\frac45$. Indeed, if there is a common neighbor $j\in N(j_1)\cap N(j_2) \cap (G\setminus\wt G)$, then $x_{j_1}'''\cdot x_{j}'''=-\frac45$ and $x_{j_2}'''\cdot x_{j}'''=-\frac45$ imply that either $x_{j_1}'''=x_{j_2}'''$ and we are done, or $x_{j_1}'''\cdot x_{j_2}'''=\frac7{25}$ which cannot happen as then $n_{j_1,j_2}=\frac{52}{5}$ is not an integer. Otherwise, all $14$ common neighbors of $j_1$ and $j_2$ are in $\wt G$, so $n_{j_1,j_2}=14$ and $x_{j_1}'''\cdot x_{j_2}'''=-\frac45$.

We obtained that $x_{j_1}'''\cdot x_{j_2}''' \in \{-\frac45,1\}$ for any $j_1,j_2\in G\setminus\wt G$. This implies that $x_j'''$, $j\in G\setminus\wt G$, attains only one of two values. After a proper choice of coordinate system in the corresponding two dimensional space, these values can be written as $(1,0)$ and $(-\frac45,\frac35)$. But then clearly
$
\sum_{i\in G\setminus\wt G}x_i''\ne(0,0).
$
On the other hand, $\sum_{i\in G}x_i=\zero$ and $x_i''=(0,0)$ for $i\in \wt G$, so
$
\sum_{i\in G\setminus\wt G}x_i''=(0,0),
$
which is a contradiction that completes the proof of the lemma.
\end{proof}

\section{The case of $16$-coclique}\label{empty16}

In this section we prove that $16$-coclique cannot be an induced subgraph of $SRG(76,30,8,14)$.

\begin{proof}[Proof of Lemma~\ref{16coclique}.]
Suppose $\wt G$ is a $16$-coclique in $G$ which is a $SRG(76,30,8,14)$. By Lemma~\ref{new-rank-lemma}~(ii), $\rank B(\wt G)=\rank(\lin(\{x_i,\,i\in\wt G\}))=16$, where $x_i\in\R^{18}$ is the Euclidean representation of $i\in G$. As in the previous section, we consider orthogonal projection of $x_j$, $j\in G\setminus \wt G$ onto the orthogonal complement of $\lin\{x_i,i\in\wt G\}$ in $\R^{18}$, which is a two-dimensional Euclidean space.

For $j\in G\setminus \wt G$, denote by $x_j'$ the projection of $x_j$ onto $\lin\{x_i,i\in\wt G\}$. Using the dual Euclidean representation of $G$, namely that for any $i\in G$ there exists $z_i\in \R^{57}$ satisfying~\eqref{pq_def} with $(p,q)=(\frac1{15},-\frac1{15})$, we immediately obtain $|N(j)\cap\wt G|=|N'(j)\cap\wt G|=8$ for any $j\in G\setminus\wt G$. Following the techniques of the previous section, this allows to verify that
\[
x_j'=-\frac4{15}\sum_{i\in N(j)\cap\wt G}x_i+\frac7{30}\sum_{i\in N'(j)\cap\wt G}x_i,
\]
and
\be\label{xj12-prime-k16}
x_{j_1}' \cdot x_{j_2}'=\frac{19}{90}n_{j_1,j_2}-\frac{112}{135},
\ee
for $j,j_1,j_2\in  G\setminus\wt G$, where $n_{j_1,j_2}=|N(j_1)\cap N(j_2)\cap\wt G|$.
If $j_1=j_2$, then $n_{j_1,j_2}=8$, so all projections $x_j'':=x_j-x_j'$, $j\in G\setminus\wt G$, have the same Euclidean norm, which means they belong to a ($2$-dimensional, planar) circle. Again we define the normalized projections as $x_j''':=\frac{x_j''}{\|x_j''\|}$.

Using~\eqref{xj12-prime-k16}, if $a\in\{0,1\}$ is the number of edges (adjacency) between $j_1$ and $j_2$, then
$x_{j_1}'''\cdot x_{j_2}'''=-\frac32n_{j_1,j_2}+7-3a\in[-1,1]$,
which, as $n_{j_1,j_2}$ is integer, leads to one of the following four possibilities:
\be\label{proj20}
x_{j_1}'''\cdot x_{j_2}'''=\begin{cases}
1,&\text{if }n_{j_1,j_2}=2\text{ and }a=1,\\
-\frac12,&\text{if }n_{j_1,j_2}=3\text{ and }a=1,\\
1,&\text{if }n_{j_1,j_2}=4\text{ and }a=0,\\
-\frac12,&\text{if }n_{j_1,j_2}=5\text{ and }a=0.
\end{cases}
\ee
In particular, $x_{j_1}'''\cdot x_{j_2}'''\in\{1,-\frac12\}$, so that there are only three possible values for $x_j'''$, $j\in G\setminus\wt G$, which are the vertices of an equilateral triangle inscribed into the unit circle. Now let $\{H_1,H_2,H_3\}$ be the partition of $G\setminus\wt G$ such that the value of $x_j'''$ is the same for any $j$ in one component of the partition. Without loss of generality, $x_j'''=(\cos(2t\pi/3),\sin(2t\pi/3))$, $j\in H_t$, $t=1,2,3$. Arguing as in the end of the proof of Lemma~\ref{srg40},
we have $\sum_{j\in G\setminus\wt G}x_j'''=(0,0)$, which implies $|H_1|=|H_2|=|H_3|=20$.

It is sufficient to work with $H_1$, but the same statements are valid for the other two components of the partition. First we show that $H_1$ is $2$-regular. For any $i\in\wt G$, we have $N(i)\subset G\setminus\wt G$ and $|N(i)|=30$. We claim that $|N(i)\cap H_1|=10$. By applying projections to~\eqref{vertex-through-neighbors}, we have that $\sum_{j\in N(i)}x_j'''=(0,0)$, therefore there is equal number of elements of $N(i)$ in each part $H_1$, $H_2$, and $H_3$. Hence, $|N(i)\cap H_1|=10$. Note that by~\eqref{proj20}, two different vertices from $H_1$ have either two or four common neighbors in $\wt G$ if they are adjacent or non-adjacent, respectively. Computing the number of (non-oriented) paths of length $2$ originating and terminating in $H_1$ going through $\wt G$, we obtain that there are $20$ edges in $H_1$. Using the Euclidean representation of $G$ in $\R^{57}$, we verify that $(\sum_{i\in H_1}z_i)^2=0$, so for any $j\in H_1$ the equation $z_j\cdot (\sum_{i\in H_1}z_i)=0$ implies $|N(j)\cap H_1|=2$, and $H_1$ is $2$-regular.

Any $2$-regular graph is a union of cycles. Next we show that if $C_l$ is a cycle of length $l$ in $H_1$, then for any $i\in\wt G$, we have $|N(i)\cap C_l|=l/2$, in particular $l$ is even and is not less than $4$. We know that $|N(i)\cap H_1|=10$, so if $H_1$ consists only of one cycle, we are done. Otherwise, it is enough to show for any two cycles $C_{l_1}$ and $C_{l_2}$ in $H_1$ of lengths $l_1$ and $l_2$ respectively, we have $l_1/l_2=|N(i)\cap C_{l_1}|/|N(i)\cap C_{l_2}|$. Let $a_t=|N(i)\cap C_{l_t}|$, $t=1,2$. Recall that for a subset $A$ of vertices of $G$, we set $\br{A}:=\sum_{i\in A}x_i$. It is straightforward to verify that $(l_2 \br{C_{l_1}}-l_1\br{C_{l_2}})^2=0$, and then $0=x_i\cdot(l_2 \br{C_{l_1}}-l_1\br{C_{l_2}})=(l_2a_1-l_1a_2)p$
yields the desired $l_1/l_2=a_1/a_2$.

 Since all projections $x_j''$, $j\in H_1$, are the same, and they are projections onto a $2$-dimensional subspace of $\R^{18}$, we have $\rank(\lin(\{x_j,\,j\in H_1\}))\le 17$, so by Lemma~\ref{new-rank-lemma}~(iv), there are at least $4$ cycles in $H_1$. Therefore, there are only the following three possibilities for the lengths of the cycles: five cycles of length $4$, or two cycles of length $6$ and two cycles of length $4$, or one cycle of length $8$ and three cycles of length $4$. In any of the cases, there is a cycle $C_4\subset H_1$ of length $4$, which will suffice for us to complete the proof.

Suppose that $\wt G=\{g_1,\dots,g_{16}\}$. For $i\in H_1$, define $A(i)$ as the $8$-element subset of $\{1,2,\dots,16\}$ such that $N(i)\cap\wt G=\{g_t:t\in A(i)\}$. By~\eqref{proj20}, if $i,j\in H_1$ are adjacent, then $|A(i)\cap A(j)|=2$; and if $i,j\in H_1$ are non-adjacent, then $|A(i)\cap A(j)|=4$. It is not hard to see that without loss of generality (by permutation of indexes) we can assume that our $C_4$ has the following representation:
\begin{gather*}
\{A(i):i\in C_4\}=\{\{1,2,3,4,5,6,7,8\},\{1,2,9,{10},{11},{12},{13},{14}\},\\
\{5,6,7,8,{13},{14},{15},{16}\},\{3,4,9,{10},{11},{12},{15},{16}\}\}.
\end{gather*}
Now let $\M$ be the collection of all $8$-element subsets of $\{1,2,\dots,16\}$, then $|\M|=\binom{16}8=12870$. Consider the following graph on $\M$: two vertices $A_1,A_2\in\M$ are adjacent if and only if $|A_1\cap A_2|$ is either $2$ or $4$. We fix $\M_0:=\{A(i):i\in C_4\}$, $|\M_0|=4$, and define $\M_1:=\{A\in\M:\M_0\subset N(A)\}$, where $N(A)$ denotes all neighbors of $A$ in our graph on $\M$. Clearly, $\{A(i):i\in H_1\setminus C_4\}$ is a $16$-clique in $\M_1$. We obtain a contradiction by showing that the largest clique in $\M_1$ has size $15$.

To this end, we use the mathematical software Sage, in particular, the function {\tt clique\_number} returning the order of the largest clique of the given graph, which is based on the Bron-Kerbosch algorithm~\cite{Br-Ke}. Note that $\M_1$ can be easily generated, it has $906$ vertices and $176672$ edges. The procedure's running time is well under one hour on a modern personal computer. See~\cite{BPR-web} for the source code and the output.
\end{proof}

\begin{remark}
One can use the second cycle of length four to reduce the problem to graphs of smaller size that would not require the use of the more sophisticated algorithms for the computation of the largest clique. However, this would lead to a somewhat more complicated programming and longer running time.
\end{remark}

\section{The case of $K_{6,10}$}\label{K6,10}

In this section we prove that $K_{6,10}$ cannot be an induced subgraph of $SRG(76,30,8,14)$.

\begin{proof}[Proof of Lemma~\ref{K610}.]
Let $\wt G$ is a $K_{6,10}$ and a subgraph of $G$, which is $SRG(76,30,8,14)$. By Lemma~\ref{new-rank-lemma}~(iii), $\rank B(\wt G)=\rank(\lin(\{x_i,\,i\in\wt G\}))=15$, where $x_i\in\R^{18}$ is the Euclidean representation of $i\in G$. As in the previous sections, we consider orthogonal projection of $x_j$, $j\in G\setminus \wt G$ onto the orthogonal complement of $\lin\{x_i,i\in\wt G\}$ in $\R^{18}$, which is now a three-dimensional Euclidean space.

Let $\wt G_1$ be the $6$-coclique in $\wt G$, and $\wt G_2$ be the $10$-coclique in $\wt G$, so that $\wt G=\wt G_1 \cup \wt G_2$. For any $j\in G\setminus \wt G$, we claim that $|N(j)\cap \wt G_1|=2$ and $|N(j)\cap \wt G_2|=4$. As before, for a subset $A$ of vertices of $G$, we set $\br{A}:=\sum_{i\in A}x_i$. Gram matrix of $\br{G_1}$, $\br{G_2}$ is singular, and $3\br{G_1}+2\br{G_2}=0$, which, after multiplication by $x_j$, leads to $57|N(j)\cap \wt G_1|+ 38 |N(j)\cap \wt G_2| =266$. Arguing similarly for the dual Euclidean representation, we obtain another linear equation $-|N(j)\cap \wt G_1|+|N(j)\cap \wt G_2|=2$, and the claim follows.

Following the same process as in the previous two sections (with the difference that the number of neighbors in each $\wt G_1$ and $\wt G_2$ needs to be tracked), this allows to verify that
\begin{equation}\label{proj-fla-610}
x_j'=-\frac14\sum_{i\in N(j)\cap\wt G}x_i+\frac14\sum_{i\in N'(j)\cap\wt G}x_i
\end{equation}
and
\be\label{xj12-k610}
x_{j_1}'\cdot x_{j_2}'=\frac{19}{90}n_{j_1,j_2}-\frac{43}{90},
\ee
for $j,j_1,j_2\in  G\setminus\wt G$, where $n_{j_1,j_2}=|N(j_1)\cap N(j_2)\cap\wt G|$.
If $j_1=j_2$, then $n_{j_1,j_2}=6$, so all projections $x_j'':=x_j-x_j'$, $j\in G\setminus\wt G$, have the same Euclidean norm $\|x_j''\|=\sqrt{\frac{19}{90}}$, which means they belong to a sphere in a three-dimensional Euclidean space. Again we define the normalized projections as $x_j''':=\frac{x_j''}{\|x_j''\|}$. Using~\eqref{xj12-k610},  if $j_1$ and $j_2$ are non-adjacent, then $x_{j_1}'''\cdot x_{j_2}'''=-n_{j_1,j_2}+3$. If $j_1$ and $j_2$ are adjacent, then $x_{j_1}'''\cdot x_{j_2}'''=-n_{j_1,j_2}+1$. Since $n_{j_1,j_2}$ is an integer, this implies that $x_{j_1}'''\cdot x_{j_2}'''$ can only take one of the three values from $\{-1,0,1\}$.

Therefore, it is easy to see that the possible values of $x_j'''$, $j\in G\setminus\wt G$, are vertices of an octahedron in $\R^3$, so without loss of generality we can assume that $x_j'''\in\{(\pm1,0,0),(0,\pm1,0),(0,0,\pm1)\}$, $j\in G\setminus\wt G$. Let $H_1=\{i\in G\setminus\wt G:x_i'''=(1,0,0)\}$ and $H_2=\{i\in G\setminus\wt G:x_i'''=(-1,0,0)\}$. Recall that $x_i\mapsto x_i'''$ is the normalized orthogonal projection onto the $3$-dimensional space which is the orthogonal complement of $\lin\{x_i,i\in\wt G\}$, while the whole space $\lin\{x_i,i\in G\}$ is $18$-dimensional. Therefore, 
\be\label{rank_hhg}
\rank(\lin\{x_i,i\in H_1\cup H_2\cup\wt G\})\le16.
\ee

Now we will find the number of vertices in $H_1$ and in $H_2$. Arguing as in the end of the proof of Lemma~\ref{srg40},
we have $\sum_{j\in G\setminus\wt G}x_j'''=(0,0,0)$. Then clearly $|H_1|=|H_2|$. We use the 2-design property~\eqref{design} of Euclidean representation of $G$ with $y=x_i'''$ for some fixed $i\in H_1$. Clearly, $x_j\cdot x_i'''=0$ for $j\in\wt G$. On the other hand, for $j\in G\setminus\wt G$, we have
\[
x_j\cdot x_i'''
=\begin{cases}
\sqrt{\frac{19}{90}},&\text{if }j\in H_1,\\
-\sqrt{\frac{19}{90}},&\text{if }j\in H_2,\\
0,&\text{otherwise}.
\end{cases}
\]
Thus, by~\eqref{design}, $\frac{19}{90}(|H_1|+|H_2|)=\frac{76}{18}$, so $|H_1|=|H_2|=10$.

Keeping $i\in H_1$ fixed, choose arbitrary $t\in H_1$. By~\eqref{vertex-through-neighbors} for $x_t$, we have $x_t+\frac18\sum_{j\in N(t)}x_j={0}$, and multiplying by $x_i'''$, we obtain  $|N(t)\cap H_2|=8+|N(t)\cap H_1|$ for any $t\in H_1$. Similarly, $|N(t)\cap H_1|=8+|N(t)\cap H_2|$ for any $t\in H_2$. These relations readily imply that inverting the edges between $H_1$ and $H_2$ leads to a regular graph of degree $2$ on the $20$ vertices of $H_1\cup H_2$. Any regular graph of degree $2$ is a union of cycles, which makes (computer) generation of all possible subgraphs $H_1\cup H_2$ rather straightforward. We will also utilize a very simple consequence of Euclidean representation which significantly lowers the number of subgraphs that need verification. Namely, the number $w$ of edges in $H_1$ is equal to the number of edges in $H_2$ and does not exceed $3$. The equality is clear, and the determinant of Gram matrix of $\br{H_1}$, $\br{H_2}$ is $-\frac{5776}{81}w + \frac{19760}{81}\ge0$, so $w\le 3$.

We generated programmatically all graphs $H_1\cup H_2$ satisfying $|H_1|=|H_2|=10$, $|N(t)\cap H_{3-i}|=8+|N(t)\cap H_i|$ for any $t\in H_i$, $i=1,2$, and that the number of edges in $H_1$ is equal to the number of edges in $H_2$ and does not exceed $3$. We obtained $5526$ graphs with some graphs possibly isomorphic to each other. Next, for every generated graph (which is a possible subgraph $H_1\cup H_2$ of $G$), we verify whether: (i) the rank of $(x_i\cdot x_j)_{i,j\in H_1\cup H_2}$ does not exceed $16$; (ii) the smallest eigenvalue of $(x_i\cdot x_j)_{i,j\in H_1\cup H_2}$ is non-negative. The conditions~(i) and~(ii) must be valid by~\eqref{rank_hhg} and Proposition~\ref{prop-rank}. There are only four graphs for which the above two conditions are satisfied, namely, when there are five cycles of length $4$ after edge inversion between $H_1$ and $H_2$.

To handle these four cases, we will add one more vertex to our subgraph and check the rank condition~(i) (satisfied due to~\eqref{rank_hhg}). Recall that $\wt G_1$ is the $6$-coclique of $\wt G$, which is a $K_{6,10}$. There is a vertex $t\in \wt G_1$ such that $|N(t)\cap H_1|=|N(t)\cap H_2|\le 3$. Indeed, let $t\in \wt G_1$ be arbitrary. By~\eqref{vertex-through-neighbors} for $x_t$, we have $x_t+\frac18\sum_{j\in N(t)}x_j={0}$, and multiplying by $x_i'''$, where $i\in H_1$, we obtain $|N(t)\cap H_1|=|N(t)\cap H_2|$. But recall that for any vertex $j\in H_1\cup H_2$, we have $|N(j)\cap \wt G_1|=2$, so there are $40$ edges between $H_1\cup H_2$ and $\wt G_1$. Hence, there must be $t\in\wt G_1$ with no more than $\frac{40}{6}$ neighbors in $H_1\cup H_2$, and the claim follows. Now, by computer verification of $66104$ graphs on $21$ vertices (generated by considering all choices of neighbors of $t$ in $H_1$ and $H_2$ for the four remaining subgraphs $H_1\cup H_2$), it turns out that the rank of $(x_i\cdot x_j)_{i,j\in H_1\cup H_2\cup\{t\}}$ is always at least $17$, which is a contradiction.

We want to remark that the computations required for this lemma take less than $15$ minutes on a modern personal computer.
\end{proof}

\end{document}